\documentclass[12pt,reqno]{amsart}
\usepackage{amsmath,amsthm,amssymb,amscd,url,enumerate,mathtools}
\usepackage{graphicx}
\usepackage[margin=1in]{geometry}
\usepackage{afterpage}
\usepackage[usenames,dvipsnames]{xcolor} 
\usepackage{tikz, tikz-3dplot, pgfplots}
\usepgfplotslibrary{fillbetween} 
\usepackage{todonotes}
\usepackage[all]{xy}
\usepackage[colorlinks=true,citecolor={blue},linkcolor = purple]{hyperref}
\usepackage{tabularx}

\usepackage{dutchcal}

\usepackage{stmaryrd} 

\usepackage[square]{natbib}
\setcitestyle{numbers}

\usepackage{multicol} 
\usepackage{wrapfig} 

\usepackage[pagewise]{lineno}

\newcommand{\TITLE}{Ramification in Division Fields and Sporadic Points on Modular Curves}
\newcommand{\TITLERUNNING}{}

\theoremstyle{plain}
\newtheorem{theorem}{Theorem}
\newtheorem{conjecture}[theorem]{Conjecture}

\newtheorem{lemma}[theorem]{Lemma}
\newtheorem{corollary}[theorem]{Corollary}

\theoremstyle{definition}
\newtheorem{definition}[theorem]{Definition}

\theoremstyle{remark}
\newtheorem{remark}[theorem]{Remark}

\newtheorem{example}[theorem]{Example}

\numberwithin{theorem}{section}

  {\end{list}}

  {\end{list}}

\newcommand{\tightoverset}[2]{%
  \mathop{#2}\limits^{\vbox to -.5ex{\kern-1.05ex\hbox{$#1$}\vss}}}

\newcommand{\gp}{{\mathfrak{p}}}

\newcommand{\gP}{{\mathfrak{P}}}

\def\Ocal{{\mathcal O}}

\def\pcal{{\mathcal p}}

\newcommand{\GG}{\mathbb{G}}

\newcommand{\PP}{\mathbb{P}}
\newcommand{\QQ}{\mathbb{Q}}

\newcommand{\ZZ}{\mathbb{Z}}

\newcommand{\ol}[1]{\overline{#1}}

\newcommand{\vphi}{\varphi}

\newcommand{\Aut}{\operatorname{Aut}}

\newcommand{\can}{\operatorname{can}}

\title[\TITLERUNNING]{\vspace*{-1.3cm} \TITLE}

\date{\today}

\author[Hanson Smith]{Hanson Smith}
\address{Department of Mathematics, University of Connecticut,
341 Mansfield Road U1009
Storrs, CT 06269-1009
USA}
\email{hanson.smith@uconn.edu}


\begin{document}

\sloppy 

\maketitle

\begin{abstract}
Consider an elliptic curve $E$ over a number field $K$. Suppose that $E$ has supersingular reduction at some prime $\mathfrak{p}$ of $K$ lying above the rational prime $p$. We completely classify the valuations of the $p^n$-torsion points of $E$ by the valuation of a coefficient of the $p^{\text{th}}$ division polynomial. We apply this description to find the minimum necessary ramification at $\mathfrak{p}$ in order for $E$ to have a point of exact order $p^n$. 

Using this bound we show that sporadic points on the modular curve $X_1(p^n)$ cannot correspond to supersingular elliptic curves without a canonical subgroup. We generalize our methods to $X_1(N)$ with $N$ composite.
\end{abstract}




\section{Introduction and Results}\label{results}

The goal of this paper is to fully describe the valuations of the coordinates of $p^n$-torsion points of supersingular elliptic curves via the valuation of the coefficient of $x^{(p^2-p)/2}$ in the $p^{\text{th}}$ division polynomial (Theorem \ref{Thm: Main} and Corollary \ref{Cor: SlopesCounts}) and to use this description to preclude certain supersingular elliptic curves from corresponding to sporadic points on the modular curve $X_1(N)$ (Theorems \ref{Thm: Slopes} and \ref{modapp}). 


Before describing the results more precisely, we establish our setup along with some notation. These conventions will hold unless otherwise specified. Let $E$ be an elliptic curve over a number field $K$, with $P\in E(K)$ a point of exact order $p^n$. Let $\gp$ be a prime of $K$ lying over $p\in \ZZ$ at which $E$ has good supersingular reduction. We can adjoin all the $p^n$-torsion points of $E$ to $K$ to obtain the full $p^n$-th division field $K (E [p^n ] )$. 
Write $\gP$ for a prime of $K (E [p^n ] )$ lying over $\gp$. We will be studying the ramification of $\gp$, and to a lesser extent $\gP$, over $p$.
Let $e_{\gp}$ denote the ramification index of $\gp$ over $p$ and $e_{\gP}$ the ramification index of $\gP$ over $p$.

\begin{figure}[h]
\hspace*{-.5 in}\xymatrix{
K(E[p^n]) \ar@{-}[d] & \gP \ar@{-}[d]\\
K \ar@{-}[d] & \gp \ar@{-}[d]\\
\QQ & p}
\caption{}
\protect{\label{notation}}
\end{figure} 

The ring of integers of $K$ is $\Ocal_K$ and the local field obtained by completing $K$ at $\gp$ is denoted $K_\gp$. The valuation is denoted $v$, since $p$ will be clear from context, and is normalized so that $v(p)=1$. \textbf{Warning:} This normalization is distinct from much of the other literature in this area. We have chosen it, however, because we believe it makes the proofs more clean and straightforward. We let $\pi_\gp$ be a uniformizer, i.e., $\pi_{\gp}$ generates the maximal ideal of the valuation ring of $K_{\gp}$. The residue field of $K_{\gp}$ at $\gp$ is denoted $k_{\gp}$. 
As is common, the algebraic closure of a given field will be denoted with an overbar, e.g., $\overline{k_{\gp}}$. 

Section \ref{Sec: FormalGroup} states and establishes the main result, a complete classification of the valuations of $p^n$-torsion elements of the formal group of $E$. This classification involves the \emph{canonical subgroup}, a distinguished subgroup of $E[p]$ that lifts the kernel of Frobenius, as well as \textit{higher-level canonical subgroups}; see Definitions \ref{CanonicalSubgroup} and \ref{HigherCanGen}. 

We use this classification to obtain a lower bound for ramification above $p$ which in turn gives a lower bound for the degree of $K$:

\begin{theorem}\label{Thm: Slopes}
Minimal values for $e_{\gp}$ and $e_{\gP}$ can be determined from the valuation of a coefficient of the $p^{\text{th}}$ division polynomial. The minimal values are 
\begin{equation}\label{MinRams}
e_{\gP}\geq p^{2n}-p^{2n-2} \ \ \ \text{ and } \ \ \ e_{\gp}>\varphi(p^n)=p^n-p^{n-1}.
\end{equation} 
Further, if $E$ does not have a canonical subgroup at $\gp$, then 
\begin{equation}\label{StrongerBound}
\left(p^{2n}-p^{2n-2}\right)\mid e_{\gp}.
\end{equation} 
If $E$ has a canonical subgroup at $\gp$, then $e_{\gP}>p^{2n}-p^{2n-2}$.
\end{theorem}

Another consequence of our work is that if $E$ is defined over a number field with a supersingular prime below $\gp$ that is unramified over $p$, then $E$ cannot have a canonical subgroup at $\gp$. Hence Equation \eqref{StrongerBound} holds.

Theorem \ref{Thm: Slopes} and other consequences of Section \ref{Sec: FormalGroup} are the subject of Section \ref{RamInDiv}. In Section \ref{sporadicpoints}, we use ramification to show that on $X_1(p^n)$ a subset of the supersingular locus is disjoint from the sporadic locus: 

\begin{theorem}\label{primepower}
Let $L$ be a number field and $E/L$ be an elliptic curve that is supersingular at some prime $\pcal$ of $\Ocal_L$ above $p$. 
Suppose that $E$ does not have a canonical subgroup at $\pcal$, then $j(E)$ does not correspond to a sporadic point on $X_1 (p^n )$ for any $n>0$.
\end{theorem}
Understanding ramification also provides a similar result for $X_1(N)$ with $N$ composite; this is the content of Theorem \ref{modapp}. After proving these theorems, we demonstrate via some examples how our methods can be generalized when one is interested in specific modular curves.

Previous work has addressed the valuations of the $p^n$-torsion points of an elliptic curve that is supersingular or potentially supersingular at a prime above $p$; see Lemma 4.7 of \cite{BoxallGrant} and Lemma 5.3\footnote{This lemma contains an error which Theorem \ref{Thm: Slopes} corrects. Fortunately, this error does not affect the main results of \cite{ramsup}.} of \cite{ramsup}. However, these investigations do not focus on the importance of the canonical subgroup; in fact, the term `canonical subgroup' is not mentioned in either paper. Moreover, to our knowledge all previous literature falls short of the complete classification of the valuations of these torsion points that we are able to present here. 

Another feature of our description is it presents a natural and concrete example of the phenomenon of higher-level canonical subgroups described in more general and abstract contexts in \cite{Gouvea}, \cite{Buzzard}, \cite{ConradCanonical1}, and \cite{ConradCanonical}. See Definition \ref{HigherCanGen}.


\newpage

\tableofcontents

\section{Previous Work and Motivation}\label{previouswork}


Previous work in the area we consider has a variety of different thrusts. We quickly survey three distinct lines of research where our classification of the valuations of the coordinates of $p^n$-torsion points of supersingular elliptic curves has consequences. 

Division fields of elliptic curves have a strong analogy with cyclotomic fields, the $\GG_m$ division fields. Motivated by this analogy, one can work to describe splitting, ramification, and inertia explicitly in division fields of elliptic curves. To this end, Adelmann's book \cite{Decompbook} provides a nice introduction culminating in criteria describing the decomposition of unramified primes in various division fields. An abbreviated list of other relevant work in this area includes \cite{pdiff}, \cite{CaliKraus}, \cite{split}, \cite{Kida}, \cite{minram}, \cite{abdiv}, \cite{splittor}, \cite{FreitasKraus}, and \cite{NonMonoECs}.


Given an arbitrary elliptic curve $E$ over some number field $K$, one wishes to have a description of the torsion subgroup $E(K)_\text{tors}$. This is another motivation for the study at hand. With \cite{Mazur'sThm1} and \cite{Mazur'sThm2} Mazur described $E(\QQ)_\text{tors}$. When $[K:\QQ]=2$, Kamienny, Kenku, and Momose (culminating with \cite{QuadK} and \cite{QuadKM}) describe $E(K)_\text{tors}$.

The proofs of these results rely on carefully analyzing modular curves. In the cases, there are infinite families of elliptic curves having each of the possible torsion subgroups. This means there are infinitely many points on the corresponding modular curve of the given degree. When $[K:\QQ]\geq 3$ this is no longer the case. That is, writing $d=[K:\QQ]$, there are modular curves with only finitely many degree $d$ points. For example, let $E$ be the elliptic curve with Cremona label 162b1 (and $j$-invariant $-1 \cdot 2^{-3} \cdot 3^{2} \cdot 5^{6}$). Najman \cite{Z/21Z} has shown
\[E\left(\QQ(\zeta_9)^+\right)\cong \ZZ/21\ZZ.\]
However, it is known that only finitely many isomorphism classes of elliptic curves can have this torsion subgroup over a cubic field. See Jeon, Kim, and Schweizer's paper \cite{infinitecubic} for a list of the possible torsion subgroups that can occur for infinitely many elliptic curves over a cubic field. The point on the modular curve $X_1(21)$ corresponding to $E$ is an example of a \emph{sporadic point}\footnote{Sporadic points are also called \emph{unexpected points} or \emph{exceptional points} in the literature.}. Briefly, if $D$ is the minimal degree such that a curve $X$ has infinitely many points of degree $D$, then a sporadic point is any point with degree less than $D$. 

Recently, Derickx, Etropolski, van Hoeij, Morrow, and Zureick-Brown have shown that $X_1(21)$ is the only modular curve with a cubic sporadic point \cite{CubicTorsion}. In fact, they show that the only cubic sporadic point is the one corresponding to Najman's curve. Combined with the work of Jeon, Kim, and Schweizer \cite{infinitecubic}, this classifies the possible groups $E(K)_{\text{tors}}$ when $[K:\QQ]=3$.

Thus, in order to understand $E(K)_\text{tors}$ for more exotic $K$, it seems likely we will need to have a better understanding of sporadic points on modular curves. To this end, Bourdon, Ejder, Liu, Odumodu, and Viray \cite{sporadicbound} have made an insightful investigation of sporadic $j$-invariants ($j$-invariants corresponding to a sporadic point on a modular curve $X_1(N)$ for some $N>0$). Among the many results in \cite{sporadicbound}, they have shown that, assuming Serre's uniformity conjecture, the number of sporadic $j$-invariants in a given number field is finite. The interested reader should also consult the recent papers \cite{BGRW} and \cite{BourdonNajman}.

Finally, one can ask about bounds on the size of the torsion subgroup in terms of the degree $d$. \cite{Merel} showed that there is a uniform bound for $|E(K)_\text{tors}|$ that is independent of the curve $E/K$ and depends only on $d$. Further, Merel found that if $p$ divides $|E(K)_\text{tors}|$, then
\[p\leq d^{3d^2}.\]
Oesterl\'e later improved the bound to 
\[p\leq \left(1+3^\frac{d}{2}\right)^2;\]
however, this work was unpublished. Thanks to the work of Derickx, a proof can now be found in \cite[Appendix A]{Oestbound}. Parent \cite{Parent} showed that if $E$ has a point of order $p^n$, then 
\[p^n\leq 129(5^d-1)(3d)^6.\]
It is believed that the best possible bound on $|E(K)_\text{tors}|$ should be a polynomial in $d$. Indeed, Lozano-Robledo has conjectured \cite{LRUniform}:
\begin{conjecture}\label{torsbound}
There is a constant $C$ 
such that if $E$ has a point of exact order $p^n$ over a number field of degree $d$, then 
\[\varphi(p^n)\leq C\cdot d.\]
\end{conjecture}
In \cite{LRUniform}, Lozano-Robledo makes significant strides toward this conjecture by considering ramification in the fields of definition of $p^n$-torsion points. 
In the case where $E$ has potential supersingular reduction, the results of \cite{formalgroups} and \cite{ramsup} show Conjecture \ref{torsbound} with $C=24$. Theorem \ref{Thm: Slopes} shows that when $E$ is supersingular with no canonical subgroup, 
then $\vphi(p^n)\cdot (p^n+p^{n-1})=p^{2n}-p^{2n-2}\leq d$ and Theorem \ref{Thm: Main} gives minimal ramification in all other cases of supersingular reduction. 

\section*{Acknowledgments}

The author would like to thank \'{A}lvaro Lozano-Robledo for the helpful correspondence and discussions, Katherine Stange for the advice and direction, and Alex Braat for discovering a false claim in an earlier version of this paper. The author is also indebted to David Grant for the insightful conversations and suggestions.
This research was partially supported by NSF-CAREER CNS-1652238 under the supervision of PI Dr. Katherine E. Stange.

\section{Background on Division Polynomials}\label{divpolybackground}

Keeping the same assumptions on $E$ as in Section \ref{results}, we write a Weierstrass equation

\begin{equation*}
E:y^2+a_1xy+a_3y=x^3+a_2x^2+a_4x+a_6.
\end{equation*}

One can define division polynomials, $\Psi_n\in \ZZ[a_1,a_2,a_3,a_4,a_6,x,y]$, recursively starting with
\[\Psi_1=1, \ \ \ \ \Psi_2=2y+a_1x+a_3, \ \ \ \ \Psi_3=3x^4+b_2x^3+3b_4x^2+3b_6x+b_8,\]
\[\Psi_4=\Psi_2\left(2x^6+b_2x^5+5b_4x^4+10b_6x^3+10b_8x^2+\left(b_2b_8-b_4b_6\right)x+\left(b_4b_8-b_6^2\right)\right),\]
and using the formulas 
\[\Psi_{2m+1} =  \Psi_{m+2}\Psi_m^3-\Psi_{m-1}\Psi_{m+1}^3 \  \text{ for} \ \ m\geq 2  \ \text{ and}\]
\[\Psi_{2m}\Psi_2 = \Psi_{m-1}^2\Psi_m\Psi_{m+2}-\Psi_{m-2}\Psi_m\Psi_{m+1}^2 \ \text{ for } \ m\geq 3.\]

For a reference see \cite[Exercise 3.7]{AEC}. If $m$ is odd, we can write

\begin{equation*}\label{proddef}
\dfrac{1}{m}\Psi_m=\prod\limits_P(x-x(P)),
\end{equation*}
where the product is over the non-trivial $m$-torsion points with distinct $x$-coordinates. If $m$ is even and not 2, we have
\begin{equation*}\label{proddefeven}
\dfrac{2}{m\Psi_2}\Psi_m=\prod\limits_P(x-x(P)),
\end{equation*}
where now the product is over the non-trivial $m$-torsion points with distinct $x$-coordinates that are not 2-torsion points.
Since $E[m]\cong \ZZ/m\ZZ\times \ZZ/m\ZZ$, this definition makes it clear that when $m$ is odd $\Psi_m$ has degree $\frac{m^2-1}{2}$. The even division polynomials also have degree $\frac{m^2-1}{2}$, so long as we think of $y$ as having degree $\frac{3}{2}$ in $x$. 



It is convenient to have the valuations of the constant coefficients of various division polynomials. 

\begin{lemma}\label{constantval} 
Keep the same assumptions as in Section \ref{results}. If $p$ is an odd prime, write $c_0$ for the constant coefficient of $\Psi_{p^n}$. If $p=2$, write $c_0$ for the constant coefficient of $\Psi_2\Psi_{2^n}$. Then $v (c_0 )=0$.
\end{lemma}

\begin{proof}
Assume $p$ is odd and choose some $Q\in E (\ol{k_\gp} )$ with $x(Q)=0$. Note that $\alpha$ is a root of $\Psi_{p^n}$ if and only if points on $E$ with $x$-coordinates equal to $\alpha$ are $p^n$-torsion points. Thus $\Psi_{p^n}(0)=c_0\equiv 0$ modulo $\gp$ if and only if $Q$ has order dividing $p^n$. Since $E$ has supersingular reduction at $\gp$, all of the $p^n$-torsion points are in the kernel of reduction. Thus $Q$ does not have order dividing $p^n$ and $v_\gp(c_0)=0$. For $p=2$, repeat the above, but replace $\Psi_{p^n}$ with $\Psi_2\Psi_{2^n}$.
\end{proof}

Note that since $\Psi_{p^n}/\Psi_{p^{n-1}}$ has leading coefficient $p$, the valuations of the distinct $x$-coordinates of points of exact order $p^n$ sum to negative one. That is, letting $E [=\hspace{-.09 cm }p^n ]$ denote the points of exact order $p^n$,
\begin{equation}\label{sumto1}
\sum_{P\in E\left[=p^n\right]/\pm}v\left(x(P)\right)=-1.
\end{equation}
The exception to Equation \eqref{sumto1} is when $p^n=2$, then the sum of the valuations of the roots of $\Psi_2^2$ is
\begin{equation}\label{sumto1at2}
\sum_{P\in E\left[=2\right]}v\left(x(P)\right)=-2.
\end{equation}


\begin{remark}\label{ReciprocalEisenstein}
Without referencing the formal group, we can already use Lemma \ref{constantval} to obtain partial results. Let $E$ be an elliptic curve over a number field $K$ with an unramified prime $\gp$ over $p$ at which $E$ is supersingular. To employ a ``reciprocal Eisenstein" trick, consider the polynomial 
\[x^{-\frac{p^2-1}{2}}\Psi_{p}\left(x^{-1}\right)\in K_{\gp}[x].\]
Since $\gp$ is unramified and $\Psi_p$ has leading coefficient $p$, one can show this polynomial is Eisenstein. Hence, with some additional consideration of the $y$-coordinate, adjoining a point of exact order $p$ to $K_{\gp}$ is a totally ramified extension of degree $p^{2}-1$. For $p^n$-torsion, we can employ this same trick with 
\[x^{-\frac{p^{2n}-p^{2n-2}}{2}}\frac{\Psi_{p^n}}{\Psi_{p^{n-1}}}\left(x^{-1}\right).\]
Though this result is significantly weaker than Theorem \ref{Thm: Main}, it is actually all that is necessary to prove Theorem \ref{modapp}.
\end{remark}

\section{Valuations of the $p^n$-torsion Points of the Formal Group}\label{Sec: FormalGroup}

In this section we will fully describe the valuations of the $p^n$-torsion points of an elliptic curve that is supersingular at a prime above $p$. All the facts used here regarding formal groups can be found in \cite[Chapter IV]{AEC}. For the case when the ramification index above $p$ is 1, see Serre's landmark paper \cite{Serre}. 

\subsection{Setup and Basics}

Unless otherwise noted, in this section $P\in E(K)$ will be a point of exact order $p^n$ with $n\geq 1$ on an elliptic curve $E/K$ that is supersingular at some prime $\gp\subset \Ocal_K$ lying above the odd prime $p$; for the caveat with $p=2$, see Remark \ref{p=2caveat}. As before, $v$ is the valuation associated to $\gp$, normalized so that $v(p)=1$. Recall, $E [=\hspace{-.09 cm }p^n ]$ denotes the set of points of exact order $p^n$.

Let $F_E\in K_{\gp}\llbracket S, T \rrbracket$ denote the formal group law of $E$ over $K_{\gp}$. Let $\pi_{\gp}$ be a uniformizer at $\gp$; the set of elements of the ideal $(\pi_{\gp})$ becomes a group under 
$F_E$ and will be denoted $\hat{E}$. If $\beta \in \hat{E}$, the map $\beta\mapsto (x (\beta ),y (\beta ))$ yields an injection from $\hat{E}$ into $E (K_{\gp} )$. Conversely, if $(x,y)\in E (K_{\gp} )$ is in the kernel of reduction, then $(x,y)\mapsto \frac{-x}{y}\in \hat{E}$ yields an inverse, so that $\hat{E}$ is isomorphic to the kernel of reduction modulo $\pi_{\gp}$. 
If $Q$ is in the kernel of reduction, write $\hat{Q}$ for the image of $Q$ in $\hat{E}$, i.e., $\frac{-x(Q)}{y(Q)}.$ We will often just use $+$ to denote addition in $\hat{E}$. The reader may find it useful to recall the following identity for moving between valuations of roots of division polynomials and of elements of the formal group:
\begin{equation*}
v(x(Q))=-2v\left(\hat{Q}\right).
\end{equation*}

Consider $E [p^n ]$. Since $E$ is supersingular, $E [p^n ]\cong \hat{E} [p^n ]$. We will conflate these groups when context makes our meaning clear. Multiplication by an integer relatively prime to $p$ is an automorphism of $\hat{E}[p]$. Choosing a basis $ \{B_1,B_2 \}$, the action of $ (\ZZ/p\ZZ )^*$ has $p+1$ orbits:
\[\left<B_1\right>, \left<B_1+B_2\right>, \dots, \left<B_1+[p-1] B_2\right>, \left<B_2\right>.\]
We label the orbits $C_0$ through $C_p$. 
Excluding the identity, we notice $v (\hat{Q} )$ is the same for all $\hat{Q}$ in a given orbit.

We have two possibilities:  In one case, 
 $v (\hat{Q} )$ is the same for all $\hat{Q}\in \hat{E}[=\hspace{-.09 cm }p]$. In the other case, 
 $v (\hat{R} )>v (\hat{Q} )$ for all $\hat{R}$ in some orbit $C_i$ and all $\hat{Q}$ in some other orbit $C_j$. 
We notice $v (\hat{S} )=v (\hat{Q} )$ if $\hat{S}$ is in any orbit $C_k$ with $k\neq i$.  This is because $\hat{S}=[l]\hat{R}+[m]\hat{Q}$ for some $l,m\in\ZZ/p\ZZ$ with $m\neq 0$, and $F_E(S,T)=S+T +(\text{terms of degree }\geq 2).$
\begin{definition}\label{CanonicalSubgroup}
The orbit $C_i$ along with the identity is called the \emph{canonical subgroup}\footnote{The nomenclature comes from the fact that such a subgroup is a canonical lifting of the kernel of Frobenius.}. That is, if there exists $\hat{R}\in \hat{E}[=\hspace{-.09 cm}p]$ such that $v (\hat{R} )>v (\hat{Q} )$ for some $\hat{Q}\in \hat{E}[=\hspace{-.09 cm }p]$, then the orbit $ \langle\hat{R} \rangle$ is the canonical subgroup and denoted $C_{\can}$.
\end{definition}
Imprecisely speaking, we could say elliptic curves with a canonical subgroup are less supersingular since, like ordinary elliptic curves, they also have a distinguished subgroup of order $p$ that is a canonical lift of the kernel of Frobenius. 
For more general discussions of canonical subgroups, the reader should consult \cite{Lubin} and \cite{Coleman}. 

We are primarily concerned with the fact that elements in the canonical subgroup have larger valuations (equivalently, the fact that $x$-coordinates of points in $C_{\can}$ have smaller valuations). In this section, we will explicitly show that there can be a similar phenomenon at higher levels and describe exactly when and to what extent it happens. Specifically, in certain fibers $[p]^{-1}\hat{Q}$ with $\hat{Q}\in \hat{E} [=\hspace{-.09 cm }p^{n-1} ]$ there may be a set of $p$ elements with larger valuations. For an example see Figure \ref{FiberLevel9}. 
If $\hat{W}$ is one such element, the others all have the form $\hat{W}+\hat{R}$ with $\hat{R}\in C_{\can}$. Moreover, the subgroup of $\hat{E} [p^n ]$ generated by any such $\hat{W}$ is exactly the set of elements of order dividing $p^n$ with larger valuation. We call this the \textit{level-$n$ canonical subgroup}; see Definition \ref{HigherCanGen}. 

\vspace{.1 in}

For $p>2$, define 
\begin{equation}\label{eq: formu}
\mu:=v\left(c_{\frac{p^2-p}{2}} \right), \ \ \text{ where } \ \ \Psi_p(x)=px^{p^2-1}+c_{p^2-2}x^{p^2-2}+\cdots+c_1x+c_0.
\end{equation}
For $p=2$, define $\mu:=v (a_1^2+4a_2 )/2$. 
With $p=2$, the following discussion needs to be augmented slightly since $x$-coordinates of points in $E[=\hspace{-.09 cm }2]$ are distinct; see Remark \ref{p=2caveat}. Returning to the discussion at hand, label the distinct $x$-coordinates of points in $E[=\hspace{-.09 cm}p]$ as $x_1,\dots, x_{\frac{p^2-1}{2}}$. Notice 
\begin{equation}\label{valofcancoeff}
c_{\frac{p^2-p}{2}}=p\sum_{1\leq i_1<i_2<\cdots<i_{\frac{p-1}{2}}\leq \frac{p^2-1}{2}} x_{i_1}\cdots x_{i_{\frac{p-1}{2}}}.
\end{equation}

If there is a canonical subgroup, there is a distinct summand on the right-hand side of \eqref{valofcancoeff} with smallest valuation. (Recall that we multiply valuations by $-2$ to move between $\hat{E}$ and $x$-coordinates of points on $E$.) Indeed the summand with smallest valuation is the product of the $\frac{p-1}{2}$ distinct $x$-coordinates of points in $C_{\can}$.
Hence, if we have a canonical subgroup, 
$-(1-\mu)$ is the sum of the valuations of all the distinct $x$-coordinates of points that are in $C_{\can}$. We see the valuation of the $x$-coordinate of a point in $C_{\can}$ is 
$\frac{-2(1-\mu)}{p-1}$. Hence, the valuation of the $x$-coordinate of a point that is not in $C_{\can}$ is $\frac{-2\mu}{p^2-p}$. 
Note, $\mu$ is defined to correspond to valuations in $\hat{E}$. Thus, if $P\in C_{\can}$, then $v (\hat{P} )=\frac{1-\mu}{p-1}$, and if $P\notin C_{\can}$, then $v (\hat{P} )=\frac{\mu}{p^2-p}$. 

A classical criterion of Deuring \cite{DeuringCriterion} states that $E: y^2 = f(x)$ is supersingular at $\gp$ if and only if the coefficient of $x^{p-1}$ in $f(x)^{\frac{p-1}{2}}$ vanishes modulo $\gp$. One may be curious how the coefficient $c_{\frac{p^2-p}{2}}$ is related to Deuring's coefficient. Thanks to Debry \cite{Debry}, we know that they are in fact equal. Thus, for such $E$, we could just as well define $\mu$ to be the valuation of Deuring's coefficient. 

Many authors 
prefer to work directly with the multiplication-by-$p$ power series $[p]T$ as opposed to the division polynomial. 
In this case, the valuation of the coefficient of $T^p$ in $[p]T$ indicates whether or not there is a canonical subgroup. Following \cite{BoxallGrant}, we have called this valuation $\mu$. The coefficient of $T^p$ is the sum of distinct products of $p^2-p$ elements of $\hat{E}[=\hspace{-.09 cm }p]$. If there is a canonical subgroup, $\mu$ is the sum of the valuations of the $p^2-p$ elements that are not in $C_{\can}$. As such, $\mu<\frac{p}{p+1}$. Conversely, if there is not a canonical subgroup, $\mu\geq \frac{p}{p+1}$. Since all elements of $\hat{E}[=\hspace{-.09 cm }p]$ have the same valuation in this case, Equation \eqref{valofcancoeff} shows the valuation of each element must be $\frac{1}{p^2-1}$. 

Before we state the main theorem of this section, it may be useful to actually ``see" a canonical subgroup. To this end, we revisit a nice example that can be found in Lozano-Robledo's papers \cite{formalgroups} and \cite{ramsup}.
\begin{example}\label{121c2}
Let $E/\QQ$ be the elliptic curve with Cremona label 121c2. The $j$-invariant is $-11\cdot 131^3$ and the global minimal model over $\QQ$ is
\begin{equation*}
E: \ \ y^2+xy=x^3+x^2-3632+82757.
\end{equation*}
At $p=11$ the curve $E$ has bad additive reduction. Over $\QQ (\sqrt[3]{11} )$ the bad additive reduction resolves to good supersingular reduction and the curve has global minimal model 
\begin{equation*}
E: \ \ y^2+\sqrt[3]{11}xy=x^3+\sqrt[3]{11^2}x^2+3\sqrt[3]{11}+2.
\end{equation*}

Using SageMath \cite{Sage}, one can compute the factorization
\begin{align*}
\Psi_{11} &=11 x^{60} + \cdots 
+ 195530917 \sqrt[3]{11} x^{55} +\cdots - 7312712 \sqrt[3]{11} x - 303271 \\
&=11\left(x^{5} + \sqrt[3]{11^2} x^{4} + 3 \sqrt[3]{11} x^{3} + 3 x^{2} -
\frac{1}{\sqrt[3]{11^2}}\right) \left(x^{55} + \cdots + \frac{303271}{\sqrt[3]{11}}\right).
\end{align*}

The valuation of the coefficient of $x^{55}$ is $\frac{1}{3}$, so $\mu=\frac{1}{3}$. 
From the factorization above, we see that the sum of the valuations of the five distinct $x$-coordinates of 11-torsion points in the canonical subgroup is $-\frac{2}{3}=-(1-\mu)$. Figures \ref{PsiNewtonPolygon} and \ref{FormalNewtonPolygon} illustrate the various Newton polygons associated to the 11-torsion on this elliptic curve. The polygons are constructed by taking the lower convex hull of the points $ (i,v (c_i ) )$, where $c_i$ is the coefficient of $x^i$ or $T^i$. 
We can visualize the canonical subgroup in its various guises as the side of the Newton polygon with steepest slope. Figure \ref{PsiNewtonPolygon} is labeled specifically referring to our $\Psi_{11}$, while Figure \ref{FormalNewtonPolygon} is labeled generally to help clarify the larger discussion.

\begin{figure}[h!]
\centering
\begin{tikzpicture}[scale = .9]
\begin{axis}[ 
	xmin=0, xmax=61.9, 
	ymin=0, ymax=1.1,
	xtick={60}, ytick={1},
	xticklabels={},
	yticklabels={1},
	major tick length={0},
	grid=major, 
	line width=1pt, 
	axis lines=center
	] 
	\addplot [very thick, domain=55:60] {.33333+.133333333*(x-55)};
	\addplot [very thick, domain=0:55] {x*(.00606061)}; 
\end{axis}
\node [left] at (6.09,2.12) {$\left(55,\mu=\frac{1}{3}\right)$};
\node at (6.09,1.73) [circle,fill,inner sep=1.5pt]{};
\node [below] at (6.65,5.25) [circle,fill,inner sep=1.5pt]{};
\node at (0,0) [circle,fill,inner sep=1.5pt]{};
\node [left] at (0,0) {$\left(0,0\right)$};
\node [below] at (6.66,0) {$60$};
\end{tikzpicture}
\caption{The Newton polygon for the polynomial $\Psi_{11}=11\hspace{-.3 cm}\prod\limits_{P\in E[=11]/\pm}\hspace{-.1 cm}(x-x(P))$}
\protect{\label{PsiNewtonPolygon}}
\end{figure}
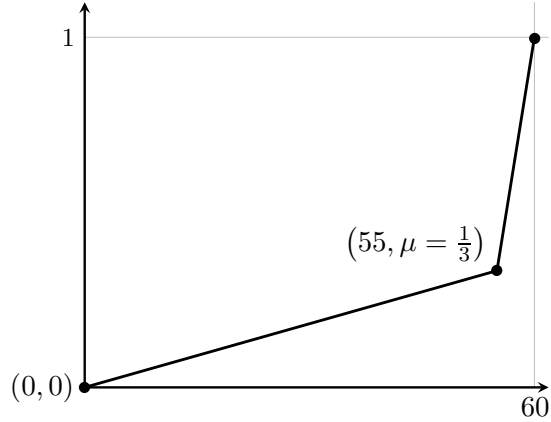


\begin{figure}[h!]
\centering
\begin{tikzpicture}[scale = 1]
\begin{axis}[ 
	xmin=0, xmax=124.9, 
	ymin=0, ymax=1.1,
	xtick={0,1,...,120}, ytick={0,0.1,...,1.1}, 
	xticklabels={0, ,, , , , , ,
, ,
 , ,},
	yticklabels={0, , , , 
	, , , , , ,1},
	major tick length={0},
	line width=1pt, 
	axis lines=center
	] 
	\addplot [very thick, domain=0:10] {1-x*.066666};  
	\addplot [very thick, domain=10:120] {.333333-.0030303*(x-10)};
\end{axis}
\node [right] at (.72,1.97) {$\left(p-1,\mu\right)$};
\node at (.57,1.75) [circle,fill,inner sep=1.5pt]{};
\node at (0,0) [circle,fill,inner sep=1.5pt]{};
\node [left] at (0,0) {$\left(0,0\right)$};
\node [below] at (6.66,0) {$p^2-1$};
\node at (0,5.17) [circle,fill,inner sep=1.5pt]{};
\node at (6.58,0) [circle,fill,inner sep=1.5pt]{};
\end{tikzpicture}
\caption{The Newton polygon for the polynomial\hspace{-.2 cm} $\prod\limits_{\hat{P}\in \hat{E}[=\hspace{-.02 cm}11]}\hspace{-.1 cm}(T-\hat{P})$ }
\protect{\label{FormalNewtonPolygon}}
\end{figure}
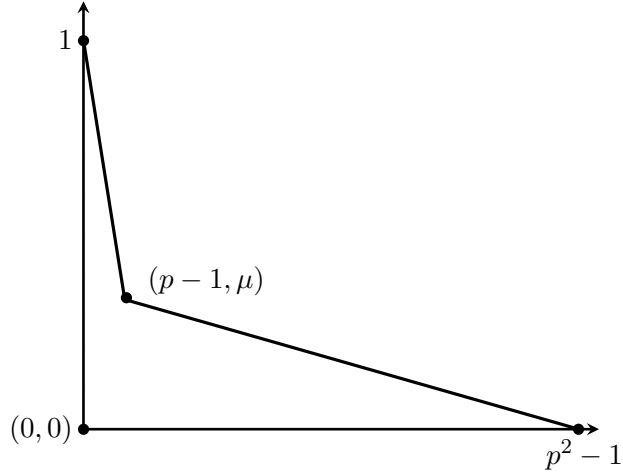




Of particular import to our work in Section \ref{sporadicpoints} is that there is a degree 10 extension of $\QQ (\sqrt[3]{11} )$ over which $E$ has a point of exact order 11. If we did not have a canonical subgroup at 11, an extension of at least degree 40 (degree $120=11^2-1$ over $\QQ$) would be needed to ensure the requisite ramification for an 11-torsion point.

\end{example}


\subsection{Main Theorem}

Echoing \cite{ConradCanonical1} and \cite{ConradCanonical}, we make the following definition. This definition helps make sense of the statement of Theorem \ref{Thm: Main} and will be justified in the proof of the theorem.

\begin{definition}\label{HigherCanGen}
Suppose $\mu<\frac{1}{p^{n-2}(p+1)}$ and $\hat{W}\in \hat{E} [=\hspace{-.09 cm }p^n ]$ is an element with the largest possible valuation, i.e., $v(\hat{W})=\frac{1-p^{n-1}\mu}{p^{n-1}(p-1)}$. As a porism of Theorem \ref{Thm: Main}, the subgroup of $\hat{E} [p^n ]$ generated by $\hat{W}$ is exactly the set of elements of $\hat{E} [p^n ]$ with large valuations. In other words, $\langle\hat{W}\rangle$ is $C_{\can}$ and every element above $C_{\can}$ with valuation $\frac{1-p^{k}\mu}{p^{k}(p-1)}$, where $1\leq k\leq n-1$. We define the \textit{level-$n$ canonical subgroup} to be $\langle\hat{W}\rangle$. We see $\langle\hat{W}\rangle=\langle\hat{R}\rangle$ for any $\hat{R}\in \hat{E} [=\hspace{-.09 cm }p^n ]$ with $v(\hat{R})=\frac{1-p^{n-1}\mu}{p^{n-1}(p-1)}$. Such a subgroup exists if and only if $\mu<\frac{1}{p^{n-2}(p+1)}$.
\end{definition}

With a way to visualize things and some terminology, we are ready for the main theorem of the section. Some of the expressions in the theorem below are not fully simplified. This is deliberate and intended to make it easier for the reader to see how the values are obtained.


 
\begin{theorem}\label{Thm: Main}
Keep the same notation as above and recall $P\in E [=\hspace{-.09 cm}p^n ]$. If 
$\mu \geq \frac{p}{p+1}$, then there is no canonical subgroup at $\gp$ and 
\begin{equation}\label{NoCanThm}
v\left(\hat{P}\right)= \frac{1}{p^{2n}-p^{2n-2}}.
\end{equation}
Otherwise, 
$0<\mu <\frac{p}{p+1}$ and there is a (level-1) canonical subgroup at $\gp$. Suppose this is the case and let $s$ be the smallest non-negative integer such that $\mu\geq\frac{1}{p^s(p+1)}$. The last level for which there is a canonical subgroup is $s+1$. If $ [p^{n-1} ]P\notin C_{\can}$, then 
\begin{equation}\label{CanNotCanThm}
v\left(\hat{P}\right)=\frac{\mu}{p^{2n-2}\left(p^{2}-p\right)}.
\end{equation}
If $ [p^{n-1} ]P\in C_{\can}$, then there are a number of cases: First, if $n=1$,
\begin{equation}\label{CanThmn=1}
v\left(\hat{P}\right)=\frac{1-\mu}{p-1}.
\end{equation} 
For $n>1$, if there is a smallest non-negative integer $m$ such that $v ( [p^m ]\hat{P} )=\frac{\mu}{p^2-p}$, then
\begin{equation}\label{HigherNoCanVal}
v\left(\hat{P}\right)=\frac{\mu}{p^{2m}(p^2-p)}.
\end{equation}
If no such integer exists, then two cases remain. If $n \geq s+2$, then there is no level-$n$ canonical subgroup and
\begin{equation}\label{HigherCanValSecond}
v\left(\hat{P}\right)=\frac{1-p^{s}\mu}{p^{2(n-(s+1))}p^s(p-1)}=\frac{1-p^{s}\mu}{p^{2n-s-2}(p-1)}.
\end{equation}
Otherwise, there is a level-$n$ canonical subgroup and $\hat{P}$ is an element of it. The valuation is
\begin{equation}\label{HigherCanVal}
v\left(\hat{P}\right)=\frac{1-p^{n-1}\mu}{p^{n-1}(p-1)}.
\end{equation}
Write $\hat{Q}\coloneqq[p]\hat{P}$. In the last case, there are $p$ elements in $[p]^{-1}\hat{Q}$ of valuation $\frac{1-p^{n-1}\mu}{p^{n-1}(p-1)}$ and $p^2-p$ elements of valuation $\frac{\mu}{p^2-p}$.
\end{theorem}


We can state a rougher but easier to digest corollary:

\begin{corollary}\label{Cor: Dividebyp}
If $[p]\hat{P}$ is not in the (possibly non-existent) level-$(n-1)$ canonical subgroup or if $\mu \geq  \frac{1}{p^{n-2}(p+1)}$ , then 
\[v\left(\hat{P}\right)=\frac{1}{p^2}v\left([p]\hat{P}\right).\] 
Otherwise, $v(\hat{P})$ has valuation 
\[\frac{1-p^{n-1}\mu}{p^{n-1}(p-1)} \ \ \ \text{ or } \ \ \ \frac{\mu}{p^2-p},\]
depending, respectively, on whether or not $\hat{P}$ is in the level-$n$ canonical subgroup.
\end{corollary}


If we change perspectives slightly and ask for counts, we obtain
\begin{corollary}\label{Cor: SlopesCounts} 
If $n\leq s+1$, then there are $p^{n-1}(p-1)$ elements of $\hat{E} [=\hspace{-.09 cm }p^n ]$ above $C_{\can}$ with valuation $\frac{1-p^{n-1}\mu}{p^{n-1}(p-1)}$ and, for all $2\leq j\leq n$, there are $p^{2(n-j)}(p^2-p)p^{j-2}(p-1)$ elements with valuation $\frac{\mu}{p^{2(n-j)}(p^2-p)}$.  

If $n> s+1$, then there are $p^{2(n-s-1)}p^s(p-1)$ elements of $\hat{E} [=\hspace{-.09 cm }p^n ]$ above $C_{\can}$ with valuation $\frac{1-p^{s}\mu}{p^{2(n-s-1)}p^s(p-1)}$ and, for all $2\leq j\leq s+1$, there are $p^{2(n-j)}(p^2-p)p^{j-2}(p-1)$ elements with valuation $\frac{\mu}{p^{2(n-j)}(p^2-p)}$.

In either case, there are exactly $p^{2(n-1)}(p^2-p)$ elements of $\hat{E} [=\hspace{-.09 cm }p^n ]$ that are not above $C_{\can}$. Each of these elements has valuation $\frac{\mu}{p^{2(n-1)}(p^2-p)}$.
\end{corollary}

We note that, if one wishes for the valuations of $x$-coordinates of $p^n$-torsion points of $E$, they need only multiply the equations in Theorem \ref{Thm: Main} by $-2$.

\begin{proof}
First, Lemma \ref{constantval} shows 
$\mu>0$. We proceed by induction on $n$. The base case, including Equation \eqref{CanThmn=1}, is established by the discussion immediately preceding Example \ref{121c2}.

For the induction step, suppose we have our result for all $k<n$. 
Define $Q \coloneqq [p]P$, and consider the power series $[p]T-\hat{Q}$ in $K_{\gp}\llbracket T \rrbracket$. If $\alpha$ is a root of $[p]T-\hat{Q}$ in $\overline{K_{\gp}}$, e.g., $\alpha=\hat{P}=\frac{-x(P)}{y(P)}$, then we see $v([p]\alpha)=v (\hat{Q} ).$ Because the height of the formal group is 2, we have 
\begin{equation}\label{eq: powerseries}
[p]T= pf(T)+\pi_{\gp}^{\mu}g \left(T^p \right) + h \left(T^{p^2} \right),
\end{equation} 
where $f,g,$ and $h$ are power series without constant coefficients and with $f'(0),g'(0),h'(0)\in K_{\gp}^*.$ 

Recall $v (\hat{Q} )<1$ by hypothesis. If $v([p]\alpha)\geq v(p\alpha)=1+v(\alpha)$,
then $v(\alpha)\leq v (\hat{Q} )-1<0$. We have a contradiction, since $\alpha\in  (\pi_{\gp} )$. Thus $v([p]\alpha)< v(p\alpha)$ and 
\begin{equation}\label{ValIneq}
v([p]\alpha)\geq \min\left(v\left(\pi_{\gp}^{\mu}\alpha^p\right),v\left(\alpha^{p^2}\right)\right)=\min\left(\mu +pv(\alpha),p^2v(\alpha)\right).
\end{equation}
Using $v(\alpha)<v ( [p^{n-1} ]\alpha )=\frac{1}{p^2-1}$, a short computation with \eqref{ValIneq} yields Equation \eqref{NoCanThm}.

Suppose now that $ [p^m ]\alpha=\hat{S}$ with $0<m<n$ and $v (\hat{S} )\leq \frac{\mu}{p^2-p}$. Suppose further that $\mu+pv(\alpha)\leq p^2v (\alpha )$, then $v(\alpha)\geq \frac{\mu}{p^2-p}\geq v (\hat{S} )$ and we have a contradiction. Thus $v(\alpha)=\frac{1}{p^{2m}}v (\hat{S} )$. This gives us Equation \eqref{CanNotCanThm} and Equation \eqref{HigherNoCanVal} in the case when $m\neq 0$.

We turn our attention to establishing Equations  \eqref{HigherCanValSecond} and \eqref{HigherCanVal}, so we assume $\mu<\frac{p}{p+1}$. Consider $\alpha\in [p]^{-1}\hat{Q}$, where $\hat{Q}\in \hat{E} [=\hspace{-.09 cm }p^{n-1} ]$ has valuation $\frac{1-p^{n-2}\mu}{p^{n-2}(p-1)}$. Suppose $\mu\geq \frac{1}{p^{n-2}(p+1)}$ and, for a contradiction, take the case that $\mu +pv(\alpha)\leq p^2v(\alpha)$. Using $v (\hat{Q} )\leq \frac{1}{p^{n-3}(p^2-1)}$, we obtain
\[\frac{1}{p^{n-1}(p^2-1)}\leq \frac{\mu}{p^2-p}\leq v(\alpha)\leq \frac{v\left(\hat{Q}\right)}{p^2}\leq \frac{1}{p^{n-1}(p^2-1)}. \]
Unless $\mu=\frac{1}{p^{n-2}(p+1)}$, we have a contradiction. In any case, using \eqref{ValIneq}, we obtain Equation \eqref{HigherCanValSecond}. To rephrase, if $\mu\geq \frac{1}{p^{n-2}(p+1)}$, then the valuations of $p^k$-th roots of $\hat{Q}$ are obtained by dividing $v (\hat{Q} )$ by $p^{2k}$ for every $k>0$. 

Suppose $\mu<\frac{1}{p^{n-2}(p+1)}$ and $v (\hat{Q} )=\frac{1-p^{n-2}\mu}{p^{n-2}(p-1)}$. For a contradiction suppose $v(\alpha)=\frac{1}{p^2}v (\hat{Q} )=\frac{1-p^{n-2}\mu}{p^{n}(p-1)}$. Using the upper bound on $\mu$, we compute
\[\mu+pv(\alpha)= \mu +p\left(\frac{1-p^{n-2}\mu}{p^{n}(p-1)}\right)< \frac{1}{p^{n-3}(p^2-1)} < p^2\left(\frac{1-p^{n-2}\mu}{p^{n}(p-1)}\right)=p^2v(\alpha). \]
Hence $\mu+pv(\alpha)<p^2v(\alpha)$ and $v([p]\alpha)=\mu+pv(\alpha)$. This contradicts our hypothesis that $v(\alpha)=\frac{1}{p^2}v (\hat{Q} )$. Therefore $\mu+pv(\alpha)\leq p^2v(\alpha)$ for all $\alpha\in [p]^{-1}\hat{Q}$. 

For at least one $\alpha\in [p]^{-1}\hat{Q}$, we will have $v([p]\alpha)=\mu +pv(\alpha)$. If this was not the case, then $v(\alpha)>\frac{1}{p^2}v (\hat{Q} )$ for every $\alpha\in [p]^{-1}\hat{Q}$. Hence $\sum_{\alpha\in [p]^{-1}\hat{Q}}v(\alpha)>v (\hat{Q} )$. However the elements of $ [p]^{-1}\hat{Q} $ are exactly the roots of $[p]T-\hat{Q}$ in the formal group 
and the valuation of the constant coefficient is $v (\hat{Q} )$. Thus there exists $\alpha \in [p]^{-1}\hat{Q}$ with $v(\alpha)=\frac{1-p^{n-1}\mu}{p^{n-1}(p-1)}$. Considering sums of the form $\alpha+\hat{S}$ with $\hat{S}\in \hat{E}[p]$, we see that there are $p$ elements in $[p]^{-1}\hat{Q}$ with valuation $\frac{1-p^{n-1}\mu}{p^{n-1}(p-1)}$ and $p^2-p$ with valuation $\frac{\mu}{p^2-p}$. Note this also establishes the $m=0$ case of Equation \ref{HigherNoCanVal}. The $p$ elements with larger valuation are the part of the level-$n$ canonical subgroup that is above $\hat{Q}$. \end{proof}

Corollary \ref{Cor: SlopesCounts} is established by a straightforward count. 

\begin{example}\label{9tors}
Consider the elliptic curve $E: y^2+\sqrt[5]{3}xy + y = x^3+\sqrt[5]{3}x^2+2x$ over $\QQ (\sqrt[5]{3} )$. This curve has good supersingular reduction at $ (\sqrt[5]{3} )$, the lone prime above 3. We can compute 
\[\Psi_3 = 3 x^{4} + \left(\sqrt[5]{3^2} + 4 \sqrt[5]{3}\right) x^{3} + \left(3 \sqrt[5]{3} + 12\right) x^{2} +3 x - \sqrt[5]{3} - 4.\] 
Thus $\mu = v (\sqrt[5]{3^2} + 4 \sqrt[5]{3} )= \frac{1}{5}<\frac{1}{3+1}$ and $E$ has a level 2 canonical subgroup at the prime $ (\sqrt[5]{3} )$.

There are 36 distinct $x$-coordinates of points in $E[=\hspace{-.09 cm }9]$. By Theorem \ref{Thm: Main}, there are 27 with $v(x(P))=-\frac{\mu}{27}=-\frac{1}{135}$, there are 6 with $v(x(P))=-\frac{\mu}{3}=-\frac{1}{15}$, and 3 with $v(x(P))=\frac{3\mu-1}{3}=-\frac{2}{15}$. These 3 with smallest valuation correspond to the level-2 canonical subgroup. Thus in $\frac{\Psi_9}{9\Psi_3}$, the coefficient of $x^{33}$ should have valuation $3\mu-1=-\frac{2}{5}$ and the coefficient of $x^{27}$ should have valuation $\mu-1=-\frac{4}{5}$. A computation verifies this.

Shifting to thinking about elements of the formal group, Figure \ref{FiberLevel9} shows the fiber $[p]^{-1}\hat{Q}$ where $\hat{Q}\in C_{\can}$. We see there are three elements of larger valuation $\frac{1}{15}$ and six elements of smaller valuation $\frac{1}{30}$.

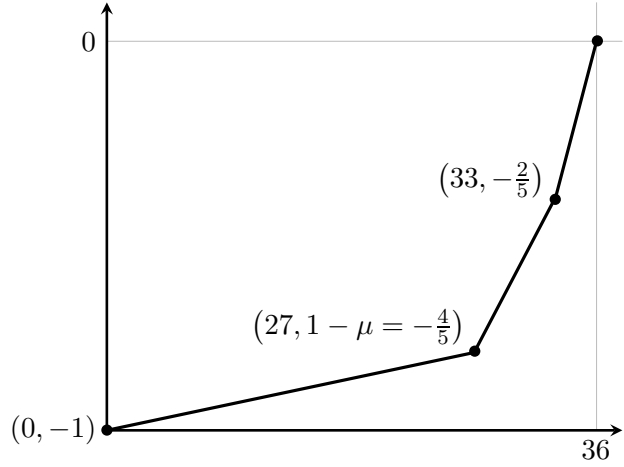
\begin{figure}[h!]
\centering
\begin{tikzpicture}[scale = 1]
\begin{axis}[ 
	xmin=0, xmax=37.9, 
	ymin=0, ymax=1.1,
	xtick={36}, ytick={1},
	xticklabels={},
	yticklabels={0},
	major tick length={0},
	grid=major, 
	line width=1pt, 
	axis lines=center
	] 
	\addplot [very thick, domain=0:27] {.0074074*x};
	\addplot [very thick, domain=27:33] {(x-27)*.066666+.2};
	\addplot [very thick, domain=33:36] {.4+.2+.133333333*(x-33)}; 
\end{axis}
\node [left] at (4.9,1.38) {$\left(27,1-\mu=-\frac{4}{5} \right)$};
\node [below] at (6.52,5.26) [circle,fill,inner sep=1.5pt]{};
\node at (0,0) [circle,fill,inner sep=1.5pt]{};
\node [left] at (0,0) {$\left(0,-1\right)$};
\node [below] at (6.5,0) {$36$};
\node at (5.96,3.07) [circle,fill,inner sep=1.5pt]{};
\node at (4.89,1.05) [circle,fill,inner sep=1.5pt]{};
\node [below] at (5.1,3.7) {$\left(33,-\frac{2}{5}\right)$};
\end{tikzpicture}
\caption{The Newton polygon for the polynomial $\Psi_{9}/\left(9\Psi_3\right)$}
\protect{\label{PsiNewtonPolygon9prim}}
\end{figure}

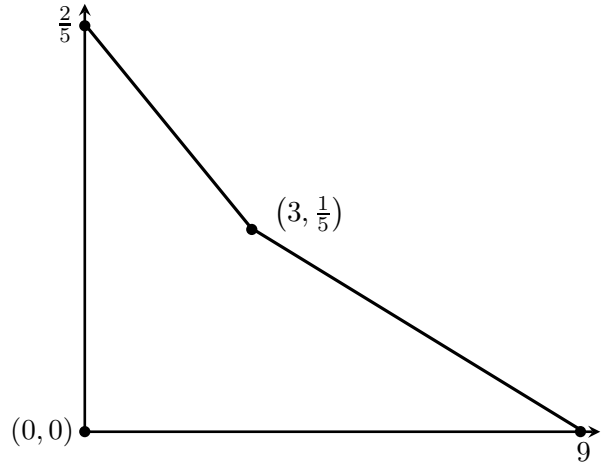
\begin{figure}[h!]
\centering
\begin{tikzpicture}[scale = 1]
\begin{axis}[ 
	xmin=0, xmax=9.3, 
	ymin=0, ymax=.42,
	xtick={9}, ytick={.4},
	xticklabels={9},
	yticklabels={$\frac{2}{5}$},
	major tick length={0},
	line width=1pt, 
	axis lines=center
	] 
	\addplot [very thick, domain=0:3] {.4-.066666*x};
	\addplot [very thick, domain=3:9] {-(x-3)*.0333333+.2}; 
\end{axis}
\node [left] at (3.6,2.9) {$\left(3,\frac{1}{5} \right)$};
\node at (2.22,2.69) [circle,fill,inner sep=1.5pt]{};
\node at (0,0) [circle,fill,inner sep=1.5pt]{};
\node [left] at (0,0) {$\left(0,0\right)$};
\node at (6.59,0) [circle,fill,inner sep=1.5pt]{};
\node at (0,5.4) [circle,fill,inner sep=1.5pt]{};
\end{tikzpicture}
\caption{The fiber above an element of \protect{$C_{\can}$}}
\protect{\label{FiberLevel9}}
\end{figure}

\end{example}

\begin{remark}\label{p=2caveat}
Consider now the case that $p=2$. Our work above with the formal group carries through almost verbatim, but there is a slight caveat for division polynomials. The roots of $\psi_2=2y+a_1x+a_3$ are not the distinct $x$-coordinates of the non-trivial 2-torsion points. Instead we must take 
\[\Psi_2^2=4x^3+\left(4a_2+a_1^2\right)x^2+(4a_4+2a_1a_3)x+4a_6+a_3^2 = 4 \prod_{P\in E[=2]}(x-x(P)).\]
Now $\sum_{P\in E[=2]} v_2(x(P))=-2$ instead of $-1$. However, summing over distinct $x$-coordinates does not introduce a factor of 2 as the $x$-coordinates of the points in $E[=\hspace{-.09 cm }2]$ are all distinct. Similarly to above, we wish to isolate the coefficient of $\frac{1}{4}\Psi_2^2$ that captures the canonical subgroup if it exists. We see
$\eta \coloneqq v_2 (\frac{1}{4} (4a_2+a_1^2 ) )$ is the valuation of the $x$-coordinate of the single point in the canonical subgroup, if there is one. The corresponding element of $\hat{E}$ has valuation 
$-\frac{\eta}{2}$. The constant coefficient of $[2]T/T$ is 2, hence $\mu= 1+ \frac{\eta}{2}=v_2 (4a_2+a_1^2 )/2$. 
\end{remark}


\section{Ramification in Division Fields}\label{RamInDiv}

Theorem \ref{Thm: Slopes} is a quick consequence of our work in Section \ref{Sec: FormalGroup}:
\begin{proof}[Proof of Theorem \ref{Thm: Slopes}] 
Equations \eqref{MinRams} and \eqref{StrongerBound} are implied directly by Theorem \ref{Thm: Main}. 
A computation with maximal values of $\mu$, shows $e_{\gP}>p^{2n}-p^{2n-2}$ when there is a canonical subgroup. 
\end{proof}

Theorem \ref{Thm: Main} can also be used, in conjunction with a lack of ramification, to preclude the existence of a canonical subgroup.

\begin{corollary}\label{RamCan}
No elliptic curve defined over $\QQ$ has a canonical subgroup at any prime. More generally, no elliptic curve that is supersingular at $\pcal$, some unramified prime over $p$, has a canonical subgroup at $\pcal$.
\end{corollary}

\begin{proof}
Notice $0<\mu<\frac{p}{p+1}$ is required for a canonical subgroup. The existence of such a valuation requires ramification.
\end{proof}
We remark that isogenies of degree coprime to $p$ also preserve the existence or nonexistence of canonical subgroups.



\begin{definition}\label{mintorsionptfield} For a positive integer $m$, define a \emph{minimal $m$-torsion point field} of an elliptic curve over a number field to be an extension of minimal degree such that the elliptic curve has a point of order $m$ over that extension. 
\end{definition}
The following corollary describes ramification in composite level torsion point fields; it will be used in Section \ref{sporadicpoints}.

\begin{corollary}\label{compositeQ}
Let $N\in \ZZ^{>1}$. Factoring $N$ into primes, we write $N=\prod _{i=1}^kp_i^{n_i}$. Let $E$ be an elliptic curve over $\QQ$ that has supersingular reduction at all the $p_i$. If $L$ is a minimal $N$-torsion point field, then any prime above $p_i$ in $L$ has ramification index $p_i^{2n_i}-p_i^{2n_i-2}$ and
\[[L:\QQ]= \prod_{i=1}^k \left(p_i^{2n_i}-p_i^{2n_i-2}\right).\]
\end{corollary}

Though we have stated Corollary \ref{compositeQ} over $\QQ$, it is valid over any number field in which there is at least one unramified, supersingular prime over each prime dividing $N$.

\begin{proof}
Let $L_i$ be a minimal $p_i^{n_i}$-torsion point field. 
If $p_i\neq p_j$, then $p_j$ is unramified in $L_i$. Observe that the compositum of all the $L_i$ 
is a minimal $N$-torsion point field. 

For $i\neq j$, consider the compositum $L_iL_j$. Since elliptic curves with supersingular reduction over $\QQ$ cannot have a canonical subgroup, the prime $p_i$ has ramification indices divisible by $p_i^{2n_i}-p_i^{2n_i-2}$. Likewise, $p_j$ has ramification indices divisible by $p_j^{2n_j}-p_j^{2n_j-2}$. However, $p_i$ is unramified in $L_j$ and $p_j$ is unramified in $L_i$. One sees that $L_iL_j$ has degree at least $p_j^{2n_j}-p_j^{2n_j-2}$ over $L_i$ so as to attain the necessary ramification. Considering $\Psi_{p_j^{n_j}}$, the degree is also at most $p_j^{2n_j}-p_j^{2n_j-2}$. The equivalent statement holds over $L_j$. Thus $L_iL_j$ has degree $ (p_j^{2n_j}-p_j^{2n_j-2} )\cdot  (p_i^{2n_i}-p_i^{2n_i-2} )$ over $\QQ$. The situation is summarized in Figure \ref{torsfield}. Iterating the above argument, we obtain the result.  
\end{proof}

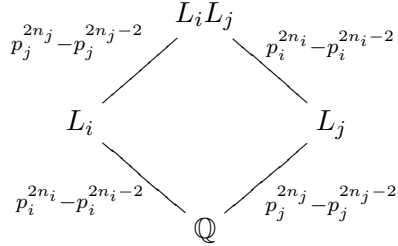
\begin{figure}[h]
\hspace*{.1in}\xymatrix{
 &L_iL_j \ar@{-}[dr]^{p_i^{2n_i}-p_i^{2n_i-2}} \ar@{-}[dl]_{p_j^{2n_j}-p_j^{2n_j-2}}&\\
L_i \ar@{-}[dr]_{p_i^{2n_i}-p_i^{2n_i-2}} & &L_j\ar@{-}[dl]^{p_j^{2n_j}-p_j^{2n_j-2}}\\
&\QQ &}
\caption{Degrees of Minimal Torsion Point Fields}
\protect{\label{torsfield}}
\end{figure}

\section{An Application to Sporadic Points on Modular Curves}\label{sporadicpoints}

Our brief exposition follows \cite{SuthNotes}. 
Let $K$ be a number field. The \emph{$K$-gonality} $\gamma_K(X)$ of a curve $X/K$ is the minimum degree among all dominant morphisms $\phi:X\to \PP^1_K$. Recall, a dominant morphism is a morphism with dense image. We will employ an upper bound on the 
$\QQ$-gonality of the modular curve $X_1(N)$. 
If $x\in X_1(N)$ is a point, define the \emph{degree} of $x$ to be the degree of the residue field at $x$ over $\QQ$. If $x$ corresponds to the data of an elliptic curve $E$ and a point $P\in E[=\hspace{-.09 cm}N]$, then the degree of $x$ is $[\QQ(j(E),\mathfrak{h}(P)):\QQ]$, where $\mathfrak{h}:E\to E/\Aut(E)\cong \PP^1$ is a Weber function for $E$. 

Given a dominant morphism $\phi:X_1(N)\to \PP^1_\QQ$ of degree $d$, one can construct infinitely many points of $X_1(N)$ defined over number fields of degree $d$ by taking preimages. 
Define $\delta(X_1(N))$ to be the smallest positive integer $k$ such that there are infinitely many points of degree $k$ on $X_1(N)$. One sees $\delta(X_1(N))\leq \gamma_\QQ(X_1(N))$. A point on $X_1(N)$ of degree strictly less than $\delta(X_1(N))$ is called a \emph{sporadic point}. Non-cuspidal, sporadic points on $X_1(N)$ correspond to finite families of elliptic curves with a point of order $N$ defined over a number field of ``small" degree. 

Now suppose $N>12$ so that $g(X_1(N))>1$. Using \cite{vanHoeijSmith}, we have the bound
\begin{equation}\label{bound}
\delta(X_1(N))\leq \gamma_\QQ\left(X_1(N)\right)\leq \dfrac{11N^2}{840}.
\end{equation} 


We apply our work in Section \ref{Sec: FormalGroup} to show Theorem \ref{primepower}, which we restate for convenience.

\vspace{.1 in}

\noindent {\bf Theorem \ref{primepower}.} \emph{
Let $L$ be a number field and $E/L$ be an elliptic curve that is supersingular at some prime $\pcal$ of $\Ocal_L$ above $p$. 
Suppose that $E$ does not have a canonical subgroup at $\pcal$, then $j(E)$ does not correspond to a sporadic point on $X_1 (p^n )$ for any $n>0$.
}\\

In what follows, we have to take special care for the CM $j$-invariants 0 and 1728. One should consult \cite{SporadicCM} for a thorough study of sporadic CM points on modular curves. The specific case of the least degree of CM points on $X_1(p^n)$ is treated in Section 6 of \cite{BourdonClark19}. 


\begin{proof}
From ramification, the minimal possible degree over $\QQ$ of an extension where an elliptic curve with $j$-invariant equal to $j(E)$ has a point $P$ of exact order $p^n$ is $\frac{1}{24} (p^{2n}-p^{2n-2} )$, where the factor $\frac{1}{24}$ comes from possibly needing to resolve bad additive reduction. For $p=2$, resolving bad additive reduction may require an extension of degree 24; for $p=3$, an extension of degree 12; and for $p>3$, an extension of degree 6. See \cite{KrausAdditive}.

First we contend with $j(E)\neq 0,1728$. Here 
\[[\QQ(j(E),\mathfrak{h}(P)):\QQ]\geq \frac{1}{48} \left(p^{2n}-p^{2n-2} \right),\]
since $|\Aut(E)|=2$. The result follows after a brief comparison with Equation \eqref{bound}.

When $j(E)=0$, then $|\Aut(E)|=6$ and when $j(E)=1728$, then $|\Aut(E)|=4$. In these cases, if we suppose $p>3$, then the minimal possible degree over $\QQ$ of an extension where $E$ has a point of exact order $p^n$ is $\frac{1}{6} (p^{2n}-p^{2n-2} )$. Now 
\[[\QQ(j(E),\mathfrak{h}(P)):\QQ]\geq \frac{1}{36} \left(p^{2n}-p^{2n-2} \right),\] 
and our previous comparison with Equation \eqref{bound} yields the result.

Take $j(E)=0$ and $p=2$. The elliptic curve $E_2: \ y^2 + y = x^3 $ is defined over $\QQ$, has $j$-invariant 0, and has good supersingular reduction at 2. Hence, if we have another elliptic curve over some number field with $j$-invariant 0 and bad additive reduction at 2, we can make a degree $|\Aut(E)|=6$ extension so that the isomorphism between the curve with bad reduction and $E_2$ is defined and thereby resolve the bad additive reduction. Thus the minimal possible degree over $\QQ$ of an extension where an elliptic curve with $j$-invariant 0 has a point of exact order $2^n$ is $\frac{1}{6} (2^{2n}-2^{2n-2} )$. Therefore,

\[[\QQ(0,\mathfrak{h}(P)):\QQ]\geq \frac{1}{36} \left(2^{2n}-2^{2n-2} \right).\] 

Take $j(E)=1728$ and $p=3$. Similarly to the above argument, the elliptic curve $E_3: \ y^2 = x^3 - x$ has $j$-invariant $1728$ and good supersingular reduction at 3. Thus we can resolve bad additive reduction at 3 by making a degree $|\Aut(E)|=4$ extension to define the isomorphism with $E_3$. We have
\[[\QQ(1728,\mathfrak{h}(P)):\QQ]\geq \frac{1}{16} \left(3^{2n}-3^{2n-2} \right).\] 

When $j(E)=1728$ and $p=2$ and when $j(E)=0$ and $p=3$, the techniques we have been employing run aground. Appealing to \citep[Theorem 3.4]{SporadicCM}, which the authors state is a special case of \cite[Theorem 7.1]{BourdonClark19} and \cite[Theorem 7.2]{BourdonClark18}, we find 
the least degree of a point on $X_1(2^n)$ corresponding to an elliptic curve with CM by $\ZZ[i]$ is $2^{2n-4}$ and 
the least degree of a point on $X_1(3^n)$ corresponding to an elliptic curve with CM by $\ZZ [\frac{1+\sqrt{-3}}{2} ]$ is $3^{2n-3}$. Respectively, these results apply to $j$-invariants 1728 and 0. Comparing with Equation \eqref{bound} finishes the proof.
\end{proof}

We can also restrict sporadic points for composite level: 
 

\begin{theorem}\label{modapp}
Let $N>12$ be a positive integer\footnote{For $1\leq N\leq 12$, the work summarized in Section \ref{previouswork} shows $X_1(N)$ has no sporadic points.}  not divisible by 6 and write $N=\prod_{i=1}^kp_i^{n_i}$ for the prime factorization. Suppose that $E/\QQ$ has good supersingular reduction at each $p_i$, then $j(E)$ does not correspond to a sporadic point on $X_1(N)$.
\end{theorem}


\begin{proof}
Corollary \ref{compositeQ} shows the degree of a minimal extension over which $E$ has a point of exact order $N$ is $\prod_{i=1}^k (p_i^{2n_i}-p_i^{2n_i-2} )$. If $E'$ is an elliptic curve over a number field $K$ with $j(E')=j(E)$ and $E'(K)[=\hspace{-.09 cm N}]\neq \emptyset$, then there is an extension of degree at most six over which $E'\cong E$. As $E$ has a point of exact order $N$ over this extension, we see $[K:\QQ]\geq \frac{1}{6}\prod_{i=1}^k (p_i^{2n_i}-p_i^{2n_i-2} )$.

Hence
\begin{equation}\label{twoterms}
[\QQ(\mathfrak{h}(P),j(E)):\QQ]\geq \frac{1}{6}\cdot \frac{1}{6}\prod_{i=1}^k\left(p_i^{2n_i}-p_i^{2n_i-2}\right)\geq \frac{1}{36}N^2-\frac{1}{36}\sum\limits_i^k \dfrac{N^2}{p_i^2},
\end{equation}
where we simply take the two leading terms for the last inequality. 

We wish to show $[\QQ(\mathfrak{h}(P),j(E)):\QQ]\geq \frac{11N^2}{840}$. Hence, via \eqref{twoterms}, our problem is implied by $\frac{1}{36}N^2-\frac{1}{36}\sum_{i=1}^k \frac{N^2}{p_i^2}\geq\frac{11N^2}{840}$.
Thus, it is enough to show 
\begin{equation}\label{primezetaestimate}
\frac{37}{2520}\geq \frac{70}{2520}\sum_{i=1}^k \frac{1}{p_i^2}.
\end{equation}
The prime zeta function evaluated at 2 is approximately $.45225$; see \href{https://oeis.org/A085548}{https://oeis.org/A085548}. 
Equation \ref{primezetaestimate} follows.
\end{proof}
 
\begin{example}\label{nonsup6}
As an example of how one might deal with non-supersingular primes, suppose $E/\QQ$ does not have supersingular reduction at any of the primes above 2 or 3. Note $j$-invariants 0 and 1728 are supersingular at 2 and 3. Take $N=6N_0$, where $\gcd (6,N_0 )=1$ and $N_0>1$, and suppose $E$ has good supersingular reduction at all the primes dividing $N_0$. The degree of the smallest extension of $K$ over which $E$ has a point $P$ of order $N$ is at least $N_0^2-\sum_{i=1}^k \frac{N_0^2}{p_i^2}$. Hence, accounting for quadratic twists with bad additive reduction, we wish to show that $ \frac{1}{2}N_0^2-\frac{1}{2}\sum_{i=1}^k \frac{N_0^2}{p_i^2}\geq \frac{11N^2}{840}=\frac{11N_0^2}{140}$.
This amounts to showing $\frac{59}{70}\approx 0.84286\geq \sum_{i=1}^k \frac{1}{p_i^2}$. Using the a rough estimate 
with the prime zeta function as above, the inequality is clear. 

Though we have shown that sporadic points on $X_1 (6N_0 )$ do not correspond to rational elliptic curves with good supersingular reduction at the prime divisors of $N_0$, the point of this example is to show that being supersingular at some primes dividing the level is an obstacle to corresponding to a sporadic point.  
\end{example}

In Example \ref{nonsup6}, we needed a lack of ramification at at least one of the primes above each of the primes dividing $N_0$ in order to preclude a canonical subgroup. The following example shows that, with or without a (level-1) canonical subgroup, it is difficult for supersingular elliptic curves to correspond to sporadic points. 

\begin{example}\label{quadratic}
Let $p\leq 17$ be a prime and let $E$ be an elliptic curve over a quadratic field that has good supersingular reduction (with or without a canonical subgroup) at a prime above $p$. 
Combining Theorem \ref{primepower} and a computation with $\mu=\frac{1}{2}$ similar to what we have done above, we can see that $j(E)$ does not correspond to a sporadic point on $X_1 (p^n )$ for any $n>0$. 

To elaborate on the computation, Theorem \ref{primepower} covers the case where $E$ does not have a canonical subgroup. Thus we must deal with the case where $\mu=\frac{1}{2}$. Here we have a canonical subgroup, but we have higher level canonical subgroups. 
In particular, if $\hat{P}$ is a $p^n$-torsion element with $n>1$, then the valuation of $\hat{P}$ is obtained by dividing the valuation of $[p]\hat{P}$ by $p^2$. The valuation of any $\hat{P}\in C_{\can}$ is $\frac{1}{2(p-1)}$, and the computation reduces to showing $\frac{1}{4}\cdot \frac{1}{2}\cdot 2(p-1)$ is greater than $\frac{11 p^2}{840}$. Here we have the factor of $\frac{1}{4}$ for potentially needing to resolve bad additive reduction and the factor of $\frac{1}{2}$ for the Weber function quotienting by $\Aut(E)$. We are making use of the fact that if $E'/L$ is another elliptic curve with $j(E')=j(E)$, then we can make a quadratic extension to define $E$ and another quadratic extension so $E'\cong E$ in order to resolve bad additive reduction. 
\end{example}

\bibliography{BibRamSpor}

\newcommand{\etalchar}[1]{$^{#1}$}
\begin{thebibliography}{BGRW20}

\bibitem[Ade01]{Decompbook}
Clemens Adelmann.
\newblock {\em The decomposition of primes in torsion point fields}, volume
  1761 of {\em Lecture Notes in Mathematics}.
\newblock Springer-Verlag, Berlin, 2001.

\bibitem[BC19]{BourdonClark19}
Abbey {Bourdon} and Pete~L. {Clark}.
\newblock {Torsion points and isogenies on CM elliptic curves}.
\newblock {\em arXiv e-prints}, June 2019.

\bibitem[BC20]{BourdonClark18}
Abbey Bourdon and Pete~L. Clark.
\newblock Torsion points and {G}alois representations on {CM} elliptic curves.
\newblock {\em Pacific J. Math.}, 305(1):43--88, 2020.

\bibitem[BEL{\etalchar{+}}19]{sporadicbound}
Abbey {Bourdon}, {\"O}zlem {Ejder}, Yuan {Liu}, Frances {Odumodu}, and Bianca
  {Viray}.
\newblock On the level of modular curves that give rise to isolated
  {$j$}-invariants.
\newblock {\em Adv. Math.}, 357:106824, 33, 2019.

\bibitem[BG03]{BoxallGrant}
John Boxall and David Grant.
\newblock Singular torsion points on elliptic curves.
\newblock {\em Math. Res. Lett.}, 10(5-6):847--866, 2003.

\bibitem[BGRW20]{BGRW}
Abbey Bourdon, David~R. Gill, Jeremy Rouse, and Lori~D. Watson.
\newblock Odd degree isolated points on $x_1(n)$ with rational $j$-invariant,
  2020.

\bibitem[BN21]{BourdonNajman}
Abbey Bourdon and Filip Najman.
\newblock Sporadic points of odd degree on $x_1(n)$ coming from
  $\mathbb{Q}$-curves, 2021.

\bibitem[Buz03]{Buzzard}
Kevin Buzzard.
\newblock Analytic continuation of overconvergent eigenforms.
\newblock {\em J. Amer. Math. Soc.}, 16(1):29--55, 2003.

\bibitem[Cen16]{splittor}
Tommaso~Giorgio Centeleghe.
\newblock Integral {T}ate modules and splitting of primes in torsion fields of
  elliptic curves.
\newblock {\em Int. J. Number Theory}, 12(1):237--248, 2016.

\bibitem[CGPS]{SporadicCM}
Pete~L. Clark, Tyler Genao, Paul Pollack, and Frederick Saia.
\newblock The least degree of a {CM} point on a modular curve.

\bibitem[CK02]{CaliKraus}
\'Elie Cali and Alain Kraus.
\newblock Sur la {$p$}-diff\'erente du corps des points de {$l$}-torsion des
  courbes elliptiques, {$l\ne p$}.
\newblock {\em Acta Arith.}, 104(1):1--21, 2002.

\bibitem[Col96]{Coleman}
Robert~F. Coleman.
\newblock Classical and overconvergent modular forms.
\newblock {\em Invent. Math.}, 124(1-3):215--241, 1996.

\bibitem[Con]{ConradCanonical}
Brian Conrad.
\newblock Higher-level canonical subgroups in abelian varieties.

\bibitem[Con06]{ConradCanonical1}
Brian Conrad.
\newblock Modular curves and rigid-analytic spaces.
\newblock {\em Pure Appl. Math. Q.}, 2(1, Special Issue: In honor of John H.
  Coates. Part 1):29--110, 2006.

\bibitem[Deb14]{Debry}
Christophe Debry.
\newblock Beyond two criteria for supersingularity: coefficients of division
  polynomials.
\newblock {\em J. Th\'{e}or. Nombres Bordeaux}, 26(3):595--606, 2014.

\bibitem[Deu41]{DeuringCriterion}
Max Deuring.
\newblock Die {T}ypen der {M}ultiplikatorenringe elliptischer
  {F}unktionenk\"{o}rper.
\newblock {\em Abh. Math. Sem. Hansischen Univ.}, 14:197--272, 1941.

\bibitem[DEv{\etalchar{+}}20a]{CubicTorsion}
Maarten {Derickx}, Anastassia {Etropolski}, Mark {van Hoeij}, Jackson~S.
  {Morrow}, and David {Zureick-Brown}.
\newblock {Sporadic Cubic Torsion}.
\newblock {\em arXiv e-prints}, July 2020.

\bibitem[Dev20b]{Sage}
The~Sage Developers.
\newblock {\em {S}ageMath, the {S}age {M}athematics {S}oftware {S}ystem
  ({V}ersion 9.1)}, 2020.
\newblock \url{http://www.sagemath.org}.

\bibitem[DKSS17]{Oestbound}
M.~{Derickx}, S.~{Kamienny}, W.~{Stein}, and M.~{Stoll}.
\newblock {Torsion points on elliptic curves over number fields of small
  degree}.
\newblock {\em ArXiv e-prints}, July 2017.

\bibitem[DT02]{split}
W.~Duke and \'A. T\'oth.
\newblock The splitting of primes in division fields of elliptic curves.
\newblock {\em Experiment. Math.}, 11(4):555--565 (2003), 2002.

\bibitem[FK18]{FreitasKraus}
N.~{Freitas} and A.~{Kraus}.
\newblock {On the degree of the $p$-torsion field of elliptic curves over
  $\mathbb{Q}_{\ell}$ for $\ell \neq p$}.
\newblock {\em ArXiv e-prints}, April 2018.

\bibitem[GJLR16]{abdiv}
Enrique Gonz\'alez-Jim\'enez and {\'{A}}lvaro Lozano-Robledo.
\newblock Elliptic curves with abelian division fields.
\newblock {\em Math. Z.}, 283(3-4):835--859, 2016.

\bibitem[Gou88]{Gouvea}
Fernando~Q. Gouv\^{e}a.
\newblock {\em Arithmetic of {$p$}-adic modular forms}, volume 1304 of {\em
  Lecture Notes in Mathematics}.
\newblock Springer-Verlag, Berlin, 1988.

\bibitem[JKS04]{infinitecubic}
Daeyeol Jeon, Chang~Heon Kim, and Andreas Schweizer.
\newblock On the torsion of elliptic curves over cubic number fields.
\newblock {\em Acta Arith.}, 113(3):291--301, 2004.

\bibitem[Kam92]{QuadK}
S.~Kamienny.
\newblock Torsion points on elliptic curves and {$q$}-coefficients of modular
  forms.
\newblock {\em Invent. Math.}, 109(2):221--229, 1992.

\bibitem[Kid03]{Kida}
M.~Kida.
\newblock Ramification in the division fields of an elliptic curve.
\newblock {\em Abh. Math. Sem. Univ. Hamburg}, 73:195--207, 2003.

\bibitem[KM88]{QuadKM}
M.~A. Kenku and F.~Momose.
\newblock Torsion points on elliptic curves defined over quadratic fields.
\newblock {\em Nagoya Math. J.}, 109:125--149, 1988.

\bibitem[Kra90]{KrausAdditive}
Alain Kraus.
\newblock Sur le d\'{e}faut de semi-stabilit\'{e} des courbes elliptiques \`a
  r\'{e}duction additive.
\newblock {\em Manuscripta Math.}, 69(4):353--385, 1990.

\bibitem[Kra99]{pdiff}
Alain Kraus.
\newblock Sur la {$p$}-diff\'erente du corps des points de {$p$}-torsion des
  courbes elliptiques.
\newblock {\em Bull. Austral. Math. Soc.}, 60(3):407--428, 1999.

\bibitem[LR13]{formalgroups}
{\'A}lvaro Lozano-Robledo.
\newblock Formal groups of elliptic curves with potential good supersingular
  reduction.
\newblock {\em Pacific J. Math.}, 261(1):145--164, 2013.

\bibitem[LR15]{minram}
{\'{A}}lvaro Lozano-Robledo.
\newblock Division fields of elliptic curves with minimal ramification.
\newblock {\em Rev. Mat. Iberoam.}, 31(4):1311--1332, 2015.

\bibitem[LR16]{ramsup}
{\'{A}}lvaro Lozano-Robledo.
\newblock Ramification in the division fields of elliptic curves with potential
  supersingular reduction.
\newblock {\em Res. Number Theory}, 2:Art. 8, 25, 2016.

\bibitem[LR18]{LRUniform}
{\'{A}}lvaro Lozano-Robledo.
\newblock Uniform boundedness in terms of ramification.
\newblock {\em Res. Number Theory}, 4(1):4:6, 2018.

\bibitem[Lub79]{Lubin}
Jonathan Lubin.
\newblock Canonical subgroups of formal groups.
\newblock {\em Trans. Amer. Math. Soc.}, 251:103--127, 1979.

\bibitem[Maz77]{Mazur'sThm1}
B.~Mazur.
\newblock Modular curves and the {E}isenstein ideal.
\newblock {\em Inst. Hautes \'Etudes Sci. Publ. Math.}, (47):33--186 (1978),
  1977.

\bibitem[Maz78]{Mazur'sThm2}
B.~Mazur.
\newblock Rational isogenies of prime degree (with an appendix by {D}.
  {G}oldfeld).
\newblock {\em Invent. Math.}, 44(2):129--162, 1978.

\bibitem[Mer96]{Merel}
Lo\"{i}c Merel.
\newblock Bornes pour la torsion des courbes elliptiques sur les corps de
  nombres.
\newblock {\em Invent. Math.}, 124(1-3):437--449, 1996.

\bibitem[Naj16]{Z/21Z}
Filip Najman.
\newblock Torsion of rational elliptic curves over cubic fields and sporadic
  points on {$X_1(n)$}.
\newblock {\em Math. Res. Lett.}, 23(1):245--272, 2016.

\bibitem[Par99]{Parent}
Pierre Parent.
\newblock Bornes effectives pour la torsion des courbes elliptiques sur les
  corps de nombres.
\newblock {\em J. Reine Angew. Math.}, 506:85--116, 1999.

\bibitem[Ser72]{Serre}
Jean-Pierre Serre.
\newblock Propri\'et\'es galoisiennes des points d'ordre fini des courbes
  elliptiques.
\newblock {\em Invent. Math.}, 15(4):259--331, 1972.

\bibitem[Sil09]{AEC}
Joseph~H. Silverman.
\newblock {\em The arithmetic of elliptic curves}, volume 106 of {\em Graduate
  Texts in Mathematics}.
\newblock Springer, Dordrecht, second edition, 2009.

\bibitem[Smi21]{NonMonoECs}
Hanson Smith.
\newblock Non-monogenic division fields of elliptic curves.
\newblock {\em J. Number Theory}, 228:174--187, 2021.

\bibitem[Sut12]{SuthNotes}
Andrew Sutherland.
\newblock Torsion subgroups of elliptic curves over number fields.
\newblock 2012.

\bibitem[vHS21]{vanHoeijSmith}
Mark van Hoeij and Hanson Smith.
\newblock A divisor formula and a bound on the {$\Bbb{Q}$}-gonality of the
  modular curve {$X_1(N)$}.
\newblock {\em Res. Number Theory}, 7(2):Paper No. 22, 21, 2021.

\end{thebibliography}
\bibliographystyle{alpha} 

\end{document}